\documentclass[12pt]{article}

\usepackage{amsmath,amssymb,amsthm}
\usepackage{tikz}
\usetikzlibrary{decorations.pathreplacing}
\usepackage{graphicx}
\usepackage{hyperref}
\usepackage{float}
\usepackage{xcolor}
\usepackage{enumitem}
\usepackage{mathrsfs}
\usepackage{makecell}
\usepackage{booktabs}
\usepackage{multirow}
\usepackage{adjustbox}
\usepackage{dsfont}
\usepackage{cite}
\usepackage{amsmath,amssymb,amsthm}
\usepackage{xcolor,float}
\usepackage{dsfont}
\usepackage{tikz,multirow,booktabs,adjustbox,hyperref}
\usepackage{cite}

\usepackage{makecell}
\theoremstyle{plain}
\newtheorem{theorem}{Theorem}[section]

\newtheorem{lemma}[theorem]{Lemma}
\newtheorem{corollary}[theorem]{Corollary}

\theoremstyle{definition}

\newtheorem{claim}[theorem]{Claim}

\begin{document}

\begin{center}
{\Large \bf Resolving Open Problems on the Euler Sombor Index}

 \vspace{8mm}

 {\bf Kinkar Chandra Das, Jayanta Bera\footnote{Corresponding author.}}

 \vspace{9mm}

 \baselineskip=0.20in

 {\it Department of Mathematics, Sungkyunkwan University, \\
 Suwon 16419, Republic of Korea\/} \\[2mm]
 E-mail: {\tt  kinkardas2003@gmail.com,~jayantabera@g.skku.edu}

\vspace{4mm}

 (Received July 23, 2025)
 \end{center}

 \vspace{5mm}

 \baselineskip=0.20in

\begin{abstract}
Recently, the Euler Sombor index $(EUS)$ was introduced as a novel degree-based topological index. For a graph $G$, the Euler Sombor index is defined as
$$EUS(G) = \sum_{v_i v_j \in E(G)} \sqrt{d_i^2 + d_j^2 + d_i d_j},$$
where $d_i$ and $d_j$ denote the degrees of the vertices $v_i$ and $v_j$, respectively. Very recently, Khanra and Das \textbf{\bf [Euler Sombor index of trees, unicyclic and chemical graphs, \emph{MATCH Commun. Math. Comput. Chem.} \textbf{94} (2025) 525--548]} proposed several open problems concerning the Euler Sombor index. This paper completely resolves two of the most challenging problems posed therein. First, we determine the minimum value of the $EUS$ index among all unicyclic graphs of a fixed order and prescribed girth, and we characterize the extremal graphs that attain this minimum. Building on this result, we further establish the minimum $EUS$ index within the broader class of connected graphs of the same order and girth, and identify the corresponding extremal structures.
In addition, we classify all connected graphs that attain the maximum Euler Sombor index $(EUS)$ when both the order and the number of leaves are fixed.
 
\bigskip

 \noindent
\end{abstract}
\baselineskip=0.27in

\section{Introduction}

Let $G=(V,\,E)$ be a simple graph with vertex set $V(G) = \{v_1,\,v_2, \ldots,\,v_n\}$ and edge set $E(G)$, where $|V(G)|=n$ and $|E(G)|=m$. For any vertex $v_i \in V(G)$, we denote its neighborhood by $N_G(v_i) = \{v_k \in V(G) : v_i v_k \in E(G)\}$ and its degree by $d_i = |N_G(v_i)|$. Two vertices $v_i$ and $v_j$ that are not adjacent in $G$ can be joined by an edge to obtain a new graph, denoted by $G + v_i v_j$. The standard notations used in this article are as follows: $C_n$, $P_n$, and $K_n$ denote the cycle, path, and complete graph of order $n$, respectively. The \emph{girth} of a graph refers to the length of its shortest cycle. A vertex of degree one is called a \emph{pendant vertex}, and an edge incident to a pendant vertex is called a \emph{pendant edge}. A path $P = v_1 v_2 \ldots v_k$ is said to be a \emph{pendant path} if it is an induced sub-path of $G$ such that $d_1 = 1$, $d_2 = \cdots = d_{k-1} = 2$, and $d_k \geq 3$.

Chemical graph theory is an important branch of mathematical chemistry that uses graph theory to represent and study molecular structures. In this approach, molecules are viewed as graphs, where the vertices represent atoms, and the edges represent chemical bonds. This representation provides a strong mathematical framework for analyzing the structure and properties of molecules.

Among the various tools in this field, degree-based topological indices are widely used as numerical measures of molecular structure~\cite{gut04,gutman13degree,ali18sum,borovicanin17zagreb}. These indices are based on the degree of vertices, where the degree of a vertex indicates the number of bonds connected to the corresponding atom, reflecting its local connectivity. These indices effectively capture key structural features of molecules and have demonstrated strong correlations with various physical, chemical, and biological properties. The computational efficiency and predictive power of degree-based indices make them essential tools for studying quantitative structure-property relationships (QSPR), which help to design and analyze novel chemical compounds.

The Sombor index, introduced by Gutman \cite{gutman21geo}, is one of the most prominent degree-based topological indices. For a graph $G$, the Sombor index ($SO$) is defined as
\begin{eqnarray*}
SO(G) = \sum_{v_i v_j \in E(G)} \sqrt{d_i^2 + d_j^2},
\end{eqnarray*}
where $d_j$ denotes the degree of vertex $v_j$ in $G$. This index has attracted significant attention from researchers due to its mathematical complexity and chemical significance. Its mathematical properties and chemical applications have been extensively explored and continue to be an active area of research~\cite{nithyaa24,CAM3,chen24,CAM4,cruz21unibicy,cruz21trees,das21symmetry,das22trees,das21extremal,deng21moleculartrees,gutman21basic,gutman21geo,gutman24eu,liu22review,liu22tetracyclic,cruz21chemical,rather22sombor,horoldagvaa21sombor,zhang23sombor,das24open}.

Building on this foundation, Gutman et al.~\cite{gutman24elliptic} proposed a new degree-based topological index with geometric motivation, known as the elliptic Sombor index ($ESO$). It is defined as
\[
\operatorname{ESO}(G) = \sum_{v_i v_j \in E(G)} (d_i + d_j) \sqrt{d_i^2 + d_j^2}.
\]
The index has been studied for its mathematical properties and potential applications in chemical graph theory~\cite{espinal25elliptic,gutman24elliptic,shanmukha24elliptic,rada24benzenoid,tang24elliptic,ahmad25elliptic}.

Afterwards, another geometrically inspired index, called the Euler Sombor index ($EUS$), was proposed in~\cite{gutman24eu,tang24euler}, offering an alternative perspective on degree-based graph invariants. The Euler Sombor index ($EUS$) of a graph $G$ is defined as
\begin{eqnarray*}
EUS(G) = \sum_{v_i v_j \in E(G)} \sqrt{d_i^2 + d_j^2 + d_i d_j}.
\end{eqnarray*}

Recently, considerable research has been conducted on the extremal graph problem related to the $EUS$ index. For example, Khanra and Das~\cite{khanra25euler} described the extremal trees with respect to the $EUS$ index. Ren et al.~\cite{ren25eulertrees} presented a characterization of trees that maximize the $EUS$ index among all trees with a specified number of pendant vertices. Su and Tang~\cite{su25euler} classified the extremal unicyclic and bicyclic graphs for the $EUS$ index. Kizilirmak~\cite{kizilirmak25unicyclic} investigated unicyclic graphs with the lowest $EUS$ index, taking into account both the order and diameter of the graphs. For more recent results on the $EUS$ index, see~\cite{tang24euler,kizilirmak25euler,albalahi25tricyclic,tache25unicyclic}.

In~\cite{khanra25euler}, Khanra and Das posed the following three open problems concerning the $EUS$ index:
\begin{itemize}
    \item \textbf{Problem 1:} Find the extremal values and describe the extremal graphs for the $EUS$ index among all connected graphs of fixed order and given girth.
    \item \textbf{Problem 2:} Determine the extremal values and characterize the extremal graphs for the $EUS$ index among all connected graphs of fixed order and a given number of pendant vertices.
    \item \textbf{Problem 3:} Investigate the extremal values of the $EUS$ index for chemical unicyclic graphs and identify those that attain these values.
\end{itemize}

In \cite{khanra25euler}, Khanra and Das proved that among all unicyclic graphs, the cycle graph uniquely attains the minimum value of the Euler Sombor index. We restate their result below:

\begin{theorem}
Among all unicyclic graphs with $n \geq 3$ vertices, the unique graph that achieves the minimum Euler Sombor index is the cycle graph $C_n$. Moreover, the minimum value of the Euler Sombor index is $2\sqrt{3}\,n$.
\end{theorem}

From the above theorem, it follows that the cycle graph $C_n$ uniquely minimizes the Euler Sombor index among all chemical unicyclic graphs. This result provides a complete solution to Open Problem 3 from \cite{khanra25euler} concerning the minimal graph. However, the corresponding problem for the maximal case remains unresolved.

In this paper, we advance the study of the Euler Sombor index by identifying the connected graphs of a fixed order and prescribed girth that minimize the index. Furthermore, we classify the extremal graphs that attain the maximum Euler Sombor index among all connected graphs of a given order with a specified number of pendant (leaf) vertices.

\section{Main Results}
 We now compute the Euler Sombor index $(EUS)$ of $H_1$ (see, Fig. \ref{H1}).
\begin{lemma} \label{s1} Let $H_1$ be a unicyclic graph of order $n$ with girth $g\,(\leq n-2)$ and maximum degree $k+\ell+2$, where $k\,(\geq 0)$ is the number of pendant paths of length $1$ and $\ell\,(\geq 0)$ is the number of pendant paths of length at least $2$ (see, Fig. \ref{H1}). Then 
\begin{align*}
EUS(H_1)&=(\ell+2)\,\sqrt{(k+\ell+2)^2+2\,(k+\ell+2)+4}+\ell\,\sqrt{7}\\
&~~~~+k\,\sqrt{(k+\ell+2)^2+(k+\ell+2)+1}+(n-k-2\,\ell-2)\,\sqrt{12}.
\end{align*}

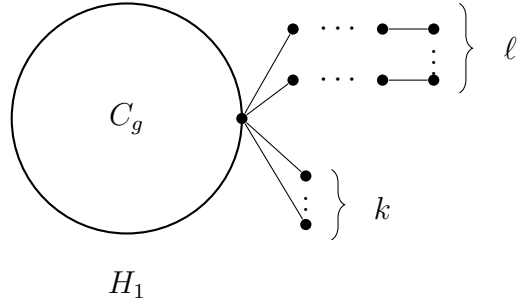
\begin{figure}[H]
\centering
\begin{tikzpicture}[scale=1.7]

\tikzset{graphnode/.style={circle, fill=black, inner sep=1.5pt}}

\draw[thick] (0,0) circle (0.9);
\node at (0,0) {\(C_g\)}; 
\node[graphnode] (c) at (0.9,0) {};

\node[graphnode] (l1a) at (1.3,0.3) {};  
\node at (1.65,0.3) {\(\cdots\)};        
\node[graphnode] (l1b) at (2.0,0.3) {};  
\node[graphnode] (l1c) at (2.4,0.3) {}; 
\draw (c) -- (l1a);          
\draw (l1b) -- (l1c);        

\node[graphnode] (l2a) at (1.3,0.7) {};
\node at (1.65,0.7) {\(\cdots\)};
\node[graphnode] (l2b) at (2.0,0.7) {};
\node[graphnode] (l2c) at (2.4,0.7) {};

\draw (c) -- (l2a);
\draw (l2b) -- (l2c);

\node at (2.4,0.5) {\(\vdots\)};
\draw[decorate,decoration={brace,amplitude=5pt}]
  (2.6,0.9) -- (2.6,0.2);
\node at (3.0, 0.55) {\(\ell\)};

\node[graphnode] (k1) at (1.4,-0.45) {};
\node[graphnode] (k2) at (1.4,-0.83) {};
\draw (c) -- (k1);
\draw (c) -- (k2);
\node at (1.4,-0.65) {\(\vdots\)};
\draw[decorate,decoration={brace,mirror,amplitude=5pt}]
  (1.6,-0.95) -- (1.6,-0.4);
\node at (2.0, -0.7) {\(k\)};

\node at (0,-1.3) {\( H_{1} \)};

\end{tikzpicture}
\caption{Unicyclic graph $H_1$.} 
\label{H1}
\end{figure}  
\end{lemma}
Denote by $C_{n,g}$ the unicyclic graph of order $n$ and girth $g$, constructed by attaching a pendant path $P_{n-g}$ to a single vertex of the cycle $C_g$. In particular, for $g=n$, $G\cong C_n$.  
\begin{corollary}\label{v0} Let $H_1$ be a graph defined in Lemma \ref{s1} (see, Fig.  \ref{H1}). Then
\begin{align}
EUS(H_1)\geq 3\,\sqrt{19}+2\,(n-4)\,\sqrt{3}+\sqrt{7}\label{v1}
\end{align}
with equality if and only if $H_1\cong C_{n,g}$.
\end{corollary}

\begin{proof} Since $g\leq n-2$, we have $k+\ell\geq 1$. First we assume that $k+\ell=1$. We must have $k=0$ and $\ell=1$ as $g\leq n-2$. Thus we have $H_1\cong C_{n,g}$ with
   $$EUS(H_1)=3\,\sqrt{19}+2\,(n-4)\,\sqrt{3}+\sqrt{7}$$
and hence the equality holds in (\ref{v1}). 

\vspace*{3mm}

Next we assume that $k+\ell\geq 2$. For $\ell=0$, we obtain
\begin{align*}
EUS(H_1)&=2\,\sqrt{(k+2)^2+2\,(k+2)+4}+k\,\sqrt{(k+2)^2+(k+2)+1}\\
&~~~~~~~~~~~~~~~~~~~~~~+(n-k-2)\,\sqrt{12}\\[2mm]
&\geq 2\,\sqrt{28}+2\,\sqrt{21}+(n-4)\,\sqrt{12}>3\,\sqrt{19}+2\,(n-4)\,\sqrt{3}+\sqrt{7}
\end{align*}
as $k\geq 2$. The inequality in (\ref{v1}) holds strictly. 

\vspace*{3mm}

Otherwise, $\ell\geq 1$. Thus we have
\begin{align*}
EUS(H_1)&=3\,\sqrt{(k+\ell+2)^2+2\,(k+\ell+2)+4}+(n-k-2\,\ell-2)\,\sqrt{12}\\[2mm]
&~~~~~~~+k\,\sqrt{(k+\ell+2)^2+(k+\ell+2)+1}+\sqrt{7}\\[2mm]
&~~~~~~~+(\ell-1)\,\Big[\sqrt{(k+\ell+2)^2+2\,(k+\ell+2)+4}+\sqrt{7}\Big]\\[2mm]
&\geq 3\,\sqrt{28}+\sqrt{7}+k\,\sqrt{21}+(n-k-2\,\ell-2)\,\sqrt{12}+(\ell-1)\\[2mm]
&~~~~~~~~~~~~~~~~\times(\sqrt{28}+\sqrt{7})\\[2mm]
&> 3\,\sqrt{28}+\sqrt{7}+k\,\sqrt{12}+(n-k-2\,\ell-2)\,\sqrt{12}+2\,(\ell-1)\,\sqrt{12}\\[2mm]
&>3\,\sqrt{19}+2\,(n-4)\,\sqrt{3}+\sqrt{7}.
\end{align*}
The inequality in (\ref{v1}) holds strictly. 
\end{proof}

We now establish a lower bound for the $EUS$ index of unicyclic graphs of order $n$ with girth $g$, and identify the extremal graphs that attain this bound.
\begin{theorem}\label{1p1}
Let $G$ be a unicyclic graph of order $n\,(\geq 3)$ with girth $g$. Then 
\begin{equation}\label{m1}
EUS(G)\geq \left\{
                    \begin{array}{ll}
                    2\sqrt{3}\,n & \hbox{for $g=n$,} \\[2mm]
                    2\sqrt{3}\,(n-3)+2\,\sqrt{19}+\sqrt{13} & \hbox{for $g=n-1$,} \\[2mm]
                    3\,\sqrt{19}+2\,(n-4)\,\sqrt{3}+\sqrt{7} & \hbox{for $g\leq n-2$}
                    \end{array}
                \right.
\end{equation}
with equality if and only if $G\cong C_n$ (for $g=n$) or $G\cong C_{n,n-1}$ (for $g=n-1$) or $G\cong C_{n,g}$ (for $g\leq n-2$).
\end{theorem}

\begin{proof} For $g=n$, we have $G\cong C_n$ with $EUS(G)=2\sqrt{3}\,n$ and hence the equality holds in (\ref{m1}). For $g=n-1$, we have $G\cong C_{n,n-1}$ with $EUS(G)=2\sqrt{3}\,(n-3)+2\,\sqrt{19}+\sqrt{13}$ and hence the equality holds in (\ref{m1}). Otherwise, $g\leq n-2$. Let $q$ be the number of vertices of degree $3$ or more. Then we have $q\geq 1$ as $g\leq n-2$. We consider the following two cases:

\vspace*{2mm}

\noindent
${\bf Case\,1}$. $q=1$. In this case there is exactly one vertex of degree $3$ or more, and all other vertices are of degree $2$ or $1$. In this case $G\cong H_1$ (see, Fig.  \ref{H1}). By Corollary \ref{v0}, we obtain
$$EUS(H_1)\geq 3\,\sqrt{19}+2\,(n-4)\,\sqrt{3}+\sqrt{7}$$
with equality if and only if $H_1\cong C_{n,g}$, that is, $G\cong C_{n,g}$.

\vspace*{2mm}

\noindent
${\bf Case\,2}$. $q\geq 2$. For each pendant path $P:\,vw_1w_2\ldots v_k$ of length at least $2$, we have 
$$\sum\limits_{v_iv_j\in E(P)}\,\sqrt{d^2_i+d^2_j+d_id_j}\geq \sqrt{19}+(k-2)\,\sqrt{12}+\sqrt{7}>k\,\sqrt{12},$$
where $d_i$ is the degree of the vertex $v_i\in V(G)$. For each pendant path of length $1$, we have 
$$\sqrt{d^2_i+d^2_j+d_id_j}\geq \sqrt{13}>\sqrt{12}.$$
Let $S$ be the set of edges of all pendant paths in $G$. Using the above results, we obtain
\begin{align}
\sum\limits_{v_iv_j\in S}\,\sqrt{d^2_i+d^2_j+d_id_j}>|S|\,\sqrt{12}.\label{b0}
\end{align}
For any edge $v_iv_j\in E(G)\backslash S$, we have 
 $$\sqrt{d^2_i+d^2_j+d_id_j}\geq \sqrt{12}.$$ 
Let
$$X=\{v_iv_j\in E(G)\backslash S\,|\,d_i\geq 3,\,d_j\geq 2\}.$$
Since $g\leq n-2$ and $q\geq 2$, then there are at least three edges in $X$, that is, $|X|\geq 3$. Thus we have
\begin{align*}
\sum\limits_{v_iv_j\in E(G)\backslash S}\,\sqrt{d^2_i+d^2_j+d_id_j}&=\sum\limits_{v_iv_j\in X}\,\sqrt{d^2_i+d^2_j+d_id_j}\\[2mm]
&~~~~~~~~+\sum\limits_{v_iv_j\in E(G)\backslash (S\cup X)}\,\sqrt{d^2_i+d^2_j+d_id_j}\\[3mm]
&\geq |X|\,\sqrt{19}+(n-|S|-|X|)\,\sqrt{12}\\[2mm]
&\geq 3\,\sqrt{19}+(n-|S|-3)\,\sqrt{12}.
\end{align*}
Using the above result with (\ref{b0}), we obtain
\begin{align*}
EUS(G)&=\sum\limits_{v_iv_j\in E(G)}\,\sqrt{d^2_i+d^2_j+d_id_j}\\[3mm]
&=\sum\limits_{v_iv_j\in E(G)\backslash S}\,\sqrt{d^2_i+d^2_j+d_id_j}+\sum\limits_{v_iv_j\in S}\,\sqrt{d^2_i+d^2_j+d_id_j}\\[3mm]
&>3\,\sqrt{19}+(n-|S|-3)\,\sqrt{12}+|S|\,\sqrt{12}\\[2mm]
&=3\,\sqrt{19}+2\,(n-3)\,\sqrt{3}>3\,\sqrt{19}+2\,(n-4)\,\sqrt{3}+\sqrt{7}. 
\end{align*}
The inequality in (\ref{m1}) holds strictly. This completes the proof of the theorem.
\end{proof}

One can easily see the following result:
\begin{lemma} \label{1r1}Let $G$ be a graph with $v_iv_j\notin E(G)$. Then $EUS(G)<EUS(G+v_iv_j)$.
\end{lemma}

\begin{theorem}
Let $G$ be a graph of order $n\,(\geq 3)$ with girth $g$. Then 
\begin{equation}\label{m11}
EUS(G)\geq \left\{
                    \begin{array}{ll}
                    2\sqrt{3}\,n & \hbox{for $g=n$,} \\[2mm]
                    2\sqrt{3}\,(n-3)+2\,\sqrt{19}+\sqrt{13} & \hbox{for $g=n-1$,} \\[2mm]
                    3\,\sqrt{19}+2\,(n-4)\,\sqrt{3}+\sqrt{7} & \hbox{for $g\leq n-2$}
                    \end{array}
                \right.
\end{equation}
with equality if and only if $G\cong C_n$ (for $g=n$) or $G\cong C_{n,n-1}$ (for $g=n-1$) or $G\cong C_{n,g}$ (for $g\leq n-2$).
\end{theorem}

\begin{proof} Let $m$ be the number of edges in the graph $G$. If $m \geq n + 1$, then there exists an edge $e \in E(G)$ such that removing $e$ results in a connected graph $G - e$ with the same girth $g$. We can iteratively apply this process--removing an appropriate edge while maintaining connectivity and girth--until we obtain a graph $H$ of order $n$ with exactly $n$ edges and girth $g$. Such a graph $H$ is necessarily unicyclic. By Lemma \ref{1r1}, we then have:
$$EUS(G) \geq EUS(H)$$
with equality if and only if $G \cong H$. 

\vspace*{3mm}

The above result with Theorem \ref{1p1}, we get the result in (\ref{m11}). 
Moreover, the equality holds if and only if $G\cong C_n$ (for $g=n$) or $G\cong C_{n,n-1}$ (for $g=n-1$) or $G\cong C_{n,g}$ (for $g\leq n-2$).
\end{proof}

\begin{lemma}{\rm \cite{KD1,KD2}}\label{w1} If $f(x)$ is a convex function with $a,\,b\geq 0$, then $f(x)-f(x-a)\geq f(x-b)-f(x-b-a)$ with equality if and only if $a$ and $b$ are both zero or one of them is zero.
\end{lemma}

Let $K_{n,p}$ denote the graph of order $n$ with $p$ pendant vertices, constructed by attaching $p$ pendant vertices to a single vertex of the complete graph $K_{n-p}$. In particular, when $p = 0$, we have $K_{n,p} \cong K_n$; and when $p = n-1$, $K_{n,p} \cong S_n$, where $S_n$ denotes the star graph of order $n$.

Let $n$ and $p$ be integers with $n>p$, and let $a_1,\,a_2,\,\ldots,\,a_{n-p}$ be non-negative integers satisfying
$$\sum_{i=1}^{n-p}\,a_i=p.$$
We define the graph
$$S(a_1,\,a_2,\,\ldots,\,a_{n-p})$$
as the graph obtained from the complete graph $K_{n-p}$ by attaching $a_i$ pendant (degree-one) vertices to the $i$-th vertex of $K_{n-p}$, for each $i = 1,\,2,\,\ldots,\,n-p$. In particular, for $p=0$, we have $a_1=a_2=\cdots=a_{n-p}=0$ and $S(a_1,\,a_2,\,\ldots,\,a_{n-p})\cong K_n$. For $a_1=p$, $a_2=a_3=\cdots=a_{n-p}=0$, $S(a_1,\,a_2,\,\ldots,\,a_{n-p})\cong K_{n,p}$. If $n-p=1$, then $S(a_1,\,a_2,\,\ldots,\,a_{n-p})\cong S_n$, where $S_n$ is a star graph of order $n$. So we assume that $n-p\geq 2$. This class of graphs is often used in extremal graph theory and in the study of degree-based topological indices. We now give an upper bound on $EUS$ index for a class of graphs of order $n$ with $p$ pendant vertices, and characterize the extremal graphs.
\begin{theorem}
Let $G$ be a graph of order $n$ with $p$ pendant vertices. Then 
\begin{align}
EUS(G)&\leq \sqrt{3}\,{n-p-1 \choose 2}\,(n-p-1)+p\,\sqrt{n^2-n+1}\nonumber\\[2mm]
&~~~~~~+(n-p-1)\,\sqrt{(n-1)\,(2n-p-2)+(n-p-1)^2}\label{1kin1}
\end{align}
with equality if and only if $G\cong K_{n,p}$.
\end{theorem}

\begin{proof} Since $G$ has $n$ vertices with $p$ pendant vertices, by Lemma \ref{1r1}, one can easily see that 
\begin{align}
EUS(G)\leq EUS(S(a_1,\,a_2,\ldots,\,a_{n-p}))\label{1t0}
\end{align}
with equality if and only if $G\cong S(a_1,\,a_2,\ldots,\,a_{n-p})$, where $a_1,\,a_2,\ldots,\,a_{n-p}$ are non-negative integers such that $a_1+a_2+\cdots+a_{n-p}=p$.
Without loss of generality, we can assume that $a_1=\max_{1\leq k\leq n-p}\,a_k$, that is, $a_1\geq a_k$ for any $1\leq k\leq n-p$. First suppose that $a_1=p,\,a_2=\ldots=a_{n-p}=0$, then $S(a_1,\,a_2,\ldots,\,a_{n-p})\cong K_{n,p}$, and hence 
\begin{align*}
&EUS(S(a_1,\,a_2,\ldots,\,a_{n-p}))=\sqrt{3}\,{n-p-1 \choose 2}\,(n-p-1)+p\,\sqrt{n^2-n+1}\\[2mm]
&~~~~~~~~~~~~~~~~~~+(n-p-1)\,\sqrt{(n-1)\,(2n-p-2)+(n-p-1)^2}.
\end{align*}
This result with (\ref{1t0}), we get the result in (\ref{1kin1}). Moreover, the equality holds in (\ref{1kin1}) if and only if $G\cong K_{n,p}$. 

Next suppose that $a_1<p$.
Let $H\cong S(a_1,\,a_2,\ldots,\,a_{n-p})$. Also let $H'\cong S(a_1+1,\,a_2,\ldots,\,a_{i-1},\,a_i-1,\,a_{i+1},\ldots,\,a_{n-p})$, where $a_i\geq 1$. Now,
{\footnotesize\begin{align}
&EUS(H')-EUS(H)\nonumber\\[2mm]
&=\sqrt{(n-p+a_1)^2+(n-p-2+a_i)^2+(n-p+a_1)\,(n-p-2+a_i)}\nonumber\\[3mm]
&-\sqrt{(n-p-1+a_1)^2+(n-p-1+a_i)^2+(n-p-1+a_1)\,(n-p-1+a_i)}\nonumber\\[3mm]
&+(a_1+1)\,\sqrt{(n-p+a_1)^2+n-p+a_1+1}-a_1\,\sqrt{(n-p-1+a_1)^2+n-p+a_1}\nonumber\\[3mm]
&+(a_i-1)\sqrt{(n-p-2+a_i)^2+n-p-1+a_i}-a_i\sqrt{(n-p-1+a_i)^2+n-p+a_i}\nonumber\\[3mm]
&+\sum\limits^{n-p}_{k=2,k\neq i}\,\Bigg[\sqrt{(n-p+a_1)^2+(n-p-1+a_k)^2+(n-p+a_1)\,(n-p-1+a_k)}\nonumber\\[3mm]
&+\sqrt{(n-p-2+a_i)^2+(n-p-1+a_k)^2+(n-p-2+a_i)\,(n-p-1+a_k)}\nonumber\\[3mm]
&-\sqrt{(n-p-1+a_1)^2+(n-p-1+a_k)^2+(n-p-1+a_1)\,(n-p-1+a_k)}\nonumber\\[3mm]
&-\sqrt{(n-p-1+a_i)^2+(n-p-1+a_k)^2+(n-p-1+a_i)\,(n-p-1+a_k)}\Bigg].\label{1t1}
\end{align}}

\vspace*{3mm}

\begin{claim}\label{c1} 
{\footnotesize\begin{align*}
&\sqrt{(n-p+a_1)^2+(n-p-2+a_i)^2+(n-p+a_1)\,(n-p-2+a_i)}\nonumber\\[3mm]
&~~~~>\sqrt{(n-p-1+a_1)^2+(n-p-1+a_i)^2+(n-p-1+a_1)\,(n-p-1+a_i)}.
\end{align*}}
\end{claim}

\vspace*{3mm}

\noindent
{\bf Proof of Claim \ref{c1}}. Since $a_1\geq a_i$, we obtain
{\footnotesize\begin{align*}
&(n-p+a_1)^2+(n-p-2+a_i)^2+(n-p+a_1)\,(n-p-2+a_i)\\[2mm]
=&(n-p-1+a_1)^2+2(n-p-1+a_1)+2+(n-p-1+a_i)^2-2(n-p-1+a_i)\\[2mm]
&+(n-p+a_1-1)\,(n-p-1+a_i)-(n-p+a_1-1)+(n-p-1+a_i)-1\\[2mm]
>&(n-p-1+a_1)^2+(n-p-1+a_i)^2+(n-p-1+a_1)\,(n-p-1+a_i).
\end{align*}}
From the above, we prove {\bf Claim \ref{c1}}.

\vspace*{3mm}

\begin{claim}\label{c2} 
{\footnotesize\begin{align*}
&(a_1+1)\,\sqrt{(n-p+a_1)^2+n-p+a_1+1}-a_1\,\sqrt{(n-p-1+a_1)^2+n-p+a_1}\\[3mm]
&+(a_i-1)\,\sqrt{(n-p-2+a_i)^2+n-p-1+a_i}-a_i\,\sqrt{(n-p-1+a_i)^2+n-p+a_i}\\[2mm]
&>0.
\end{align*}}
\end{claim}

\vspace*{3mm}

\noindent
{\bf Proof of Claim \ref{c2}}. Let us consider a function
{\footnotesize\begin{align}
f(x)&=\sqrt{(n-p-1+x)^2+n-p+x}=\sqrt{(n-p+x-1/2)^2+3/4}\nonumber\\[3mm]
&=\sqrt{(x+t)^2+3/4},\nonumber
\end{align}
where $t=n-p-1/2>0$. Then we obtain
$$f'(x)=\frac{x+t}{\sqrt{(x+t)^2+3/4}}>0~~\mbox{ and }~~f''(x)=\frac{3/4}{\Big((x+t)^2+3/4\Big)^{3/2}}>0.$$}
Thus $f(x)$ is an increasing and convex function. Setting $x=a_1+1,$ $a=1,\,b=a_1+1-a_i>0$ in Lemma \ref{w1}, we obtain
   {\footnotesize $$f(a_1+1)+f(a_i-1)>f(a_1)+f(a_i),~\mbox{ that is, }~f(a_1+1)+f(a_i-1)- f(a_1)-f(a_i)>0.$$}
Thus we have
{\footnotesize\begin{align*}
&\sqrt{(n-p+a_1)^2+n-p+a_1+1}+\sqrt{(n-p-2+a_i)^2+n-p-1+a_i}\\[2mm]
&-\sqrt{(n-p-1+a_1)^2+n-p+a_1}-\sqrt{(n-p-1+a_i)^2+n-p+a_i}\\[2mm]
&>0.
\end{align*}}
Since $a_1\geq a_i$, using the above result, we obtain
{\footnotesize\begin{align*}
&(a_1+1)\,\sqrt{(n-p+a_1)^2+n-p+a_1+1}-a_1\,\sqrt{(n-p-1+a_1)^2+n-p+a_1}\\[3mm]
&+(a_i-1)\sqrt{(n-p-2+a_i)^2+n-p-1+a_i}-a_i\sqrt{(n-p-1+a_i)^2+n-p+a_i}\\[3mm]
=&a_1\,\Big[\sqrt{(n-p+a_1)^2+n-p+a_1+1}-\sqrt{(n-p-1+a_1)^2+n-p+a_1}\,\Big]\\[3mm]
&~~~~+a_i\,\Big[\sqrt{(n-p-2+a_i)^2+n-p-1+a_i}-\sqrt{(n-p-1+a_i)^2+n-p+a_i}\,\Big]\\[3mm]
&~~~~+\sqrt{(n-p+a_1)^2+n-p+a_1+1}-\sqrt{(n-p-2+a_i)^2+n-p-1+a_i}\\[3mm]
> &a_i\,\Big[\sqrt{(n-p+a_1)^2+n-p+a_1+1}-\sqrt{(n-p-1+a_1)^2+n-p+a_1}\Big]\\[3mm]
&~~+a_i\,\Big[\sqrt{(n-p-2+a_i)^2+n-p-1+a_i}-\sqrt{(n-p-1+a_i)^2+n-p+a_i}\Big]\\[3mm]
=&a_i\,\Big[\sqrt{(n-p+a_1)^2+n-p+a_1+1}-\sqrt{(n-p-1+a_1)^2+n-p+a_1}\\[3mm]
&~~+\sqrt{(n-p-2+a_i)^2+n-p-1+a_i}-\sqrt{(n-p-1+a_i)^2+n-p+a_i}\Big]\\[3mm]
>&0.
\end{align*}}
This proves {\bf Claim \ref{c2}}.

\vspace*{3mm}

\begin{claim}\label{c3} 
{\footnotesize\begin{align*}
&\sum\limits^{n-p}_{k=2,\,k\neq i}\,\Bigg[\sqrt{(n-p+a_1)^2+(n-p-1+a_k)^2+(n-p+a_1)\,(n-p-1+a_k)}\\[3mm]
&+\sqrt{(n-p-2+a_i)^2+(n-p-1+a_k)^2+(n-p-2+a_i)\,(n-p-1+a_k)}\\[3mm]
&-\sqrt{(n-p-1+a_1)^2+(n-p-1+a_k)^2+(n-p-1+a_1)\,(n-p-1+a_k)}\\[2mm]
&-\sqrt{(n-p-1+a_i)^2+(n-p-1+a_k)^2+(n-p-1+a_i)\,(n-p-1+a_k)}\Bigg]\\[2mm]
&>0.
\end{align*}}
\end{claim}

\vspace*{3mm}

\noindent
{\bf Proof of Claim \ref{c3}}. Let us consider a function
{\footnotesize\begin{align*}
g(x)&=\sqrt{(n-p-1+x)^2+(n-p-1+a_k)^2+(n-p-1+x)\,(n-p-1+a_k)},\\[3mm]
&=\sqrt{(x+s_1)^2+s_2^2+(x+s_1)\,s_2}\,,
\end{align*}}
where $s_1=n-p-1$ and $s_2=n-p-1+a_k>0$. Then we obtain
\begin{align*}
&~~g'(x)=\frac{x+s_1+s_2/2}{\sqrt{(x+s_1)^2+s^2_2+(x+s_1)\,s_2}}>0,\\
\mbox{ and }&\\
&~~g''(x)=\frac{0.75\,s^2_2}{\Big((x+s_1)^2+s^2_2+(x+s_1)\,s_2\Big)^{3/2}}>0.
\end{align*}
Thus $g(x)$ is an increasing and convex function. Setting $x=a_1+1$, $a=1,\,b=a_1+1-a_i>0$ in Lemma \ref{w1}, we obtain
    {\footnotesize $$g(a_1+1)+g(a_i-1)>g(a_1)+g(a_i),~\mbox{ that is, }~g(a_1+1)+g(a_i-1)-g(a_1)-g(a_i)>0.$$}
Thus we have
{\footnotesize\begin{align*}
&\sqrt{(n-p+a_1)^2+(n-p-1+a_k)^2+(n-p+a_1)\,(n-p-1+a_k)}\\[2mm]
&+\sqrt{(n-p-2+a_i)^2+(n-p-1+a_k)^2+(n-p-2+a_i)\,(n-p-1+a_k)}\\[2mm]
&-\sqrt{(n-p-1+a_1)^2+(n-p-1+a_k)^2+(n-p-1+a_1)\,(n-p-1+a_k)}\\[2mm]
&-\sqrt{(n-p-1+a_i)^2+(n-p-1+a_k)^2+(n-p-1+a_i)\,(n-p-1+a_k)}>0,
\end{align*}}
that is,
{\footnotesize\begin{align*}
&\sum\limits^{n-p}_{k=2,k\neq i}\,\Bigg[\sqrt{(n-p+a_1)^2+(n-p-1+a_k)^2+(n-p+a_1)\,(n-p-1+a_k)}\\[3mm]
&+\sqrt{(n-p-2+a_i)^2+(n-p-1+a_k)^2+(n-p-2+a_i)\,(n-p-1+a_k)}\\[2mm]
&-\sqrt{(n-p-1+a_1)^2+(n-p-1+a_k)^2+(n-p-1+a_1)\,(n-p-1+a_k)}\\[2mm]
&-\sqrt{(n-p-1+a_i)^2+(n-p-1+a_k)^2+(n-p-1+a_i)\,(n-p-1+a_k)}\Bigg]\\[2mm]
&>0.
\end{align*}}
This proves {\bf Claim \ref{c3}}.

\vspace*{3mm}

Using Claims \ref{c1}, \ref{c2} and \ref{c3} in (\ref{1t1}), we obtain $EUS(H')-EUS(H)>0$, that is, $EUS(H')>EUS(H)$. Using the same transformation several times (if exists), we obtain
$$EUS(H)<EUS(H')<\cdots<EUS(S(p,\,0,\ldots,\,0))=EUS(K_{n,p}).$$
The above result with (\ref{1t0}), we obtain $EUS(G)\leq EUS(S(a_1,\,a_2,\ldots,\,a_{n-p}))$\\
$=EUS(H)<EUS(K_{n,p})$. This completes the proof of the theorem.
\end{proof}

\section{Concluding Remarks}
In this paper, we have identified the unicyclic graphs with fixed order and prescribed girth that minimize the $EUS$ index. Building on this, we extended our results to encompass all connected graphs under the same conditions, determining those that achieve the lowest index values. Moreover, we provided a characterization of connected graphs with a fixed order and specified number of pendent vertices that maximize the $EUS$ index. Moreover, we observed that, among all chemical unicyclic graphs, the cycle graph has the minimum value for the $EUS$ index. These problems were previously posed as open problems in~\cite{khanra25euler}.

However, some key problems remain open: determining the maximum $EUS$ index and characterizing the extremal graphs among connected graphs with fixed order and given girth; finding the minimum $EUS$ index and identifying the extremal graphs when both the order and number of pendent vertices are specified; and characterizing the maximal chemical unicyclic graphs for the $EUS$ index.

\vspace*{3mm}

\noindent

\end{document}